\documentclass{article}

\usepackage{amsmath,amssymb}
\usepackage{amsthm}
\usepackage{cite}
\usepackage{hyperref}
\usepackage[capitalise, noabbrev, nameinlink]{cleveref}
\usepackage{siunitx}

\usepackage{tikz}
\tikzset{
every node/.style={draw, circle, inner sep=2pt}
}

\usepackage{soul}
\usepackage{cancel}

\newtheorem{theorem}{Theorem}[section]
\newtheorem{lemma}[theorem]{Lemma}
\newtheorem{proposition}[theorem]{Proposition}
\newtheorem{corollary}[theorem]{Corollary}

\theoremstyle{definition}
\newtheorem{definition}[theorem]{Definition}
\newtheorem{observation}[theorem]{Observation}
\newtheorem{remark}[theorem]{Remark}
\newtheorem{example}[theorem]{Example}

\newtheorem{problem}[theorem]{Problem}
\newtheorem{algorithm}[theorem]{Algorithm}

\newcommand{\trans}{^\top}

\newcommand{\dunion}{\mathbin{\dot\cup}}
\newcommand{\bzero}{\mathbf{0}}
\newcommand{\bone}{\mathbf{1}}

\newcommand{\bd}{\mathbf{d}}
\newcommand{\be}{\mathbf{e}}

\newcommand{\bx}{\mathbf{x}}
\newcommand{\by}{\mathbf{y}}

\newcommand{\bu}{\mathbf{u}}
\newcommand{\bv}{\mathbf{v}}
\newcommand{\bw}{\mathbf{w}}
\newcommand{\tr}{\operatorname{tr}}
\newcommand{\nul}{\operatorname{null}}

\newcommand{\spec}{\operatorname{spec}}

\newcommand{\mptn}{\mathcal{S}}
\newcommand{\diag}{\operatorname{diag}}
\newcommand{\mult}{\operatorname{mult}}
\newcommand{\Var}{\operatorname{Var}}
\newcommand{\mv}{\operatorname{mv}}
\newcommand{\amv}{\operatorname{amv}}
\newcommand{\onemv}{\operatorname{var}_{\bone}}
\newcommand{\Paw}{\mathsf{Paw}}
\newcommand{\supp}{\operatorname{supp}}

\newcommand{\oml}{\mathbf{m}}

\title{Inverse eigenvalue problem for Laplacian matrices of a graph}
\author{
Shaun Fallat
\thanks{Department of Mathematics and Statistics, University of Regina, Regina, SK, CA (shaun.fallat@uregina.ca)}
\and
Himanshu Gupta
\thanks{Department of Mathematics and Statistics, University of Regina, Regina, SK, CA. (himanshu.gupta@uregina.ca)}
\and
Jephian C.-H.~Lin
\thanks{Department of Applied Mathematics, National Sun Yat-sen University, Kaohsiung 80424, Taiwan (jephianlin@gmail.com)}
}

\date{\today}

\begin{document}

\maketitle

\begin{abstract}
For a given graph $G$, we aim to determine the possible realizable spectra for a generalized (or sometimes referred to as a weighted) Laplacian matrix associated with $G$. This new specialized inverse eigenvalue problem is considered for certain families of graphs and graphs on a small number of vertices. Related considerations include studying the possible ordered multiplicity lists associated with stars and complete graphs and graphs with a few vertices. Finally, we present a novel investigation,  both theoretically and numerically, the minimum variance over a family of generalized Laplacian matrices with a size-normalized weighting.
\end{abstract}  


\medskip


\section{Introduction}

Let $G=(V,E)$ be a graph on $|V|=n$ vertices.  Define $\mptn(G)$ as the set of $n\times n$ real symmetric matrices $A = \begin{bmatrix} a_{i,j} \end{bmatrix}$ such that  
\[
    a_{i,j} \begin{cases}
    \neq 0 & \text{if } \{i,j\}\in E(G), \\
    = 0 & \text{if } \{i,j\}\notin E(G),\ i\neq j, \\
    \in\mathbb{R} & \text{if } i = j.
    \end{cases}
\]
The inverse eigenvalue problem for a graph (IEP-$G$) asks for all the possible spectra among matrices in $\mptn(G)$. In particular, the IEP-$G$ aims to explore algebraic and combinatorial connections between the graph and its associated spectral properties. Research on the IEP-$G$ is extensive and has produced a myriad of important results; see, e.g., the monograph \cite{IEPGbook} and the references therein. One of the motivations of the IEP-$G$ comes from vibration theory; see, e.g., Gladwell \cite{GladwellIPiV05}.  However, the inverse problems studied in vibration theory often focus only on generalized Laplacian matrices. For example, if a  spring-mass system is represented by a graph $G$, its vibration behavior is governed by the differential equation $M\ddot{\bx} = -L \bx$, where $M$ is a diagonal matrix given by the weights of the masses, $\bx$ is a vector composed of the displacements of the masses, and $L$ is a matrix in $\mptn(G)$ with nonpositive off-diagonal entries.  Therefore, we begin a study along these lines by naturally restricting our attention to such weighted matrices.  Let $\mptn_L(G)$ be the set of $n\times n$ real symmetric matrices $A = \begin{bmatrix} a_{i,j} \end{bmatrix}$ such that 
\[
    a_{i,j} \begin{cases}
    < 0 & \text{if } \{i,j\}\in E(G), \\
    = 0 & \text{if } \{i,j\}\notin E(G),\ i\neq j, \\
    - \sum_{k: k\sim i}a_{i,k} & \text{if } i = j.
    \end{cases}
\]

Analogous to the inverse eigenvalue problem for graphs (IEP-$G$), we are interested in initiating a study on an inverse eigenvalue problem among matrices in the class $\mptn_L(G)$, and abbreviate this problem as {\em IEPL} (inverse eigenvalue problem for generalized Laplacian matrices associated with a graph $G$). 
To this end, we say a collection of real numbers $0 \leq \lambda_2 \leq \lambda_3 \leq \cdots \leq \lambda_n$, is {\em Laplacian realizable} if there exists $L \in \mptn_L(G)$ with $\spec(L)=\{0, \lambda_2, \lambda_3, \ldots, \lambda_n\}$. 

Evidently, any matrix $A$ in $\mptn_L(G)$ is singular and positive semidefinite (just like the combinatorial Laplacian matrix). Moreover, if  $\{0, \lambda_2, \lambda_3, \ldots, \lambda_n\}$ is the eigenvalues of some $A$ in  $\mptn_L(G)$ for some $G$ on $n$ vertices, then it follows that the nullity of $A$ (or $\nul(A)$)  is exactly equal to the number of components of $G$, and hence $\lambda_2(A) > 0$ if and only if $G$ is connected. Along these lines, this paper studies the possible spectrum of a matrix or $\spec(A)$ for $A$ in $\mptn_L(G)$. Given a real symmetric matrix $A$, we write $\lambda_k(A)$ for its $k$-th smallest eigenvalue and $\mult_A(\lambda)$ for the multiplicity of $\lambda$ as an eigenvalue of $A$.  

Let $A$ be an $m \times n$ matrix, $\alpha\subseteq X$, and $\beta\subseteq Y$, where $|X|=m$ indexes the rows of $A$ and and $|Y|=n$ indexes the columns of $A$.  The submatrix of $A$ induced on the rows in $\alpha$ and columns in $\beta$ is denoted by $A[\alpha,\beta]$.  The submatrix of $A$ obtained by removing the rows in $\alpha$ and columns in $\beta$ is denoted by $A(\alpha,\beta)$.  When $\alpha = \beta$, we simply write $A[\alpha]$ and $A(\beta)$.  When $\alpha = \beta = \{i\}$, we write $A[i]$ and $A(i)$ to make the notation easier.  Subvectors are defined in a similar way. If $\bx$ and $\by$ are both vectors in $\mathbb{R}^n$, then we write $\bx \leq \by$ (i.e., {\em entrywise partial order}) to mean that the $i$th coordinate of $\bx$ is less than or equal to the $i$th coordinate of $\by$.

For graph notations, $P_n$, $C_n$, and $K_n$ stands for the path, the cycle, and the complete graph on $n$ vertices, respectively.  The complete bipartite graph with $p$ and $q$ vertices on each side is denoted by $K_{p,q}$.  We also use $K_4 - e$ for the graph obtained from $K_4$ by removing an edge, and $\Paw$ for the graph obtained from $K_3$ by appending a leaf.

The paper is organized as follows. We begin, in Section 2, with a treatment of the inverse eigenvalue problem for generalized Laplacian matrices associated with a graph. We consider stars and complete graphs in generality and investigate other Laplacian inverse eigenvalue problems for graphs on a small number of vertices. We then move to studies on ordered multiplicity lists for generalized Laplacian matrices in Section 3. In Section 4, we consider the minimum variance of a spectra of a certain normalized generalized Laplacian matrix and study connections to a specific quadratic programming problem.


\section{Inverse Eigenvalue Problem for Generalized Laplacian Matrices}
\label{sec:iepl}


As mentioned above the IEP-$G$ has been well studied, but generally remains unresolved for most graphs. Hence we specialize and focus our attention on characterizing the possible spectra of matrices in $\mptn_L(G)$ for a given graph.  
We begin with a simple observation for a clique (or edge) on two vertices.

\begin{observation}
\label{obs:k2}
Consider the graph $K_2$.  Then every matrix $A\in\mptn_L(K_2)$ has the form  
\[
    L = \begin{bmatrix}
        a & -a \\
        -a & a
    \end{bmatrix}
\]
for some real number $a > 0$.  It is straightforward to check that this matrix has spectrum $\{0, 2a\}$.  Therefore, for any real numbers $\{0, \lambda_2\}$ with $\lambda_2 > 0$, there is a matrix in $\mptn_L(K_2)$ with the spectrum $\{0, \lambda_2\}$.  
\end{observation}

\begin{remark}
\label{rem:NNW}
Let $G$ be a graph on $n$ vertices and $m$ edges. 
Let $N$ be its $(0,1,-1)$ vertex-edge oriented incidence matrix of some orientation of the edges of $G$.  Then the combinatorial Laplacian matrix of $G$ is $NN\trans$.  Moreover, given an edge weight assignment $\bw = (w_1, \ldots, w_m)$ with $W = \diag(\bw)$, then the corresponding generalized Laplacian matrix of $G$ is $NWN\trans$, which has the same nonzero eigenvalues as the matrices $W^{\frac{1}{2}}N\trans NW^{\frac{1}{2}}$ and $N\trans NW$.  Note that $NN\trans$ is independent of the choice of the orientation of edges, while $N\trans N$ could be different but its spectrum remains the same.
\end{remark}

Moving forward to graphs on 3 vertices, we consider the path on three vertices ($P_3$).

\begin{proposition}
\label{prop:p3}
A set of real numbers $\{0, \lambda_2, \lambda_3\}$ with $0 < \lambda_2 \leq \lambda_3$ is Laplacian realizable for $P_3$ if and only if $\lambda_3 \geq 3\lambda_2$.
\end{proposition}
\begin{proof}
Let $G = P_3$.  Then a matrix $A\in\mptn_L(G)$ has the form  
\[
    \begin{bmatrix}
        a & -a & 0 \\
        -a & a + b & -b \\
        0 & -b & b
    \end{bmatrix}
\]
for some $a,b > 0$.  Suppose $A$ has eigenvalues $\{0,\lambda_2,\lambda_3\}$ with $0 < \lambda_2 \leq \lambda_3$.  By \cref{rem:NNW}, it follows that matrix product 
\[
    \begin{bmatrix}
        2 & -1 \\
        -1 & 2
    \end{bmatrix}
    \begin{bmatrix}
        a & 0 \\
        0 & b
    \end{bmatrix}.
\]
has eigenvalues $\{\lambda_2, \lambda_3\}$.  Thus, we observe that the trace and the determinant imply the following conditions  
\[
    \begin{aligned}
        2a + 2b &= \lambda_2 + \lambda_3 = s, \\
        3ab &= \lambda_2\lambda_3 = p,
    \end{aligned}
\]
where $s$ and $p$ can be computed once $\lambda_2$ and $\lambda_3$ are given.  Hence the system of equations  
\[
    \begin{aligned}
        a + b &= \frac{1}{2}s, \\
        ab &= \frac{1}{3}p
    \end{aligned}
\]
is solvable with $a,b > 0$ if and only if 
\[
    \frac{1}{2}s \geq 2 \sqrt{\frac{1}{3}p}.
\]
Treating $\lambda_2$ as a constant to solve the inequality above for $\lambda_3$, we arrive at the desired inequality  $\lambda_3 \geq 3\lambda_2$.  
\end{proof}

In contrast, we verify that there are not so many restrictions on the possible spectra of matrices in $\mptn(K_n)$.  The following lemma is a recasting of \cite[Lemma~2.2]{Fiedler74}.

\begin{lemma}
\label{lem:join}
Let $A$ be a $p\times p$ matrix and $B$ a $q \times q$ matrix with both $A$ and $B$ being generalized Laplacian matrices, where $\spec(A) = \{0, \mu_2,\ldots,\mu_p\}$ and $\spec(B) = \{0, \tau_2, \ldots, \tau_q\}$, respectively.  Then, for $\rho>0$ the matrix 
\[
    M = \begin{bmatrix}
    A + \rho qI & -\rho J_{p,q} \\
    -\rho J_{q,p} & B + \rho pI
    \end{bmatrix} 
\]
is a generalized Laplacian matrix with spectrum  
\[
    \{0, \rho(p+q), 
    \mu_2 + \rho q, \ldots, \mu_p + \rho q, 
    \tau_2 + \rho p, \ldots, \tau_q + \rho p\}.
\]
\end{lemma}
\begin{proof}
It is straightforward to observe that $M\bone = \bzero$ and that $M$ is a generalized Laplacian matrix.  
Let $\{\bx_1, \ldots, \bx_p\}$ be an orthonormal basis of $A$ 
with $\bx_1 = \frac{1}{\sqrt{p}}\bone$
corresponding to the  eigenvalues $0, \mu_2, \ldots, \mu_p$.  Further, by direct computation 
\[
    \begin{bmatrix} \bx_i \\ \bzero \end{bmatrix}
\]
is an eigenvector of $M$ with respect to $\mu_2 + \rho q, \ldots, \mu_p + \rho q$ for $i = 2, \ldots, p$.  Similarly, given the orthonormal basis $\{\by_1, \ldots, \by_q\}$ of $B$ with $\by_1 = \frac{1}{\sqrt{q}}\bone$
corresponding to the eigenvalues $0, \tau_2, \ldots, \tau_q$, the vector 
\[
    \begin{bmatrix} \bzero \\ \by_i \end{bmatrix}
\]
is an eigenvector of $M$ with respect to eigenvalues $\tau_2 + \rho p, \ldots, \tau_q + \rho p$ for $i = 2, \ldots, q$.

By the trace condition, we determine the final eigenvalue of $M$ is $\rho(p + q)$.  This completes the proof.
\end{proof}

We now consider the IEPL for the complete graph on $n$ vertices. Note that in \cite{DLvM19}, the authors proved that an arbitrary collection $\Lambda$ can be Laplacian realizable by \emph{some} graph, but we verify that $K_n$ is a graph for which any collection $\Lambda$ is Laplacian realizable.

\begin{theorem}
\label{thm:kn}
Let $n \geq 2$ and $\Lambda = \{0, \lambda_2, \ldots, \lambda_n\}$ be a multiset of real numbers with $\lambda_2, \ldots, \lambda_n$ positive.  Then $\Lambda$ is Laplacian realizable for $K_n$.  Moreover, $K_n$ is the only connected graph that can realize all possible such collections $\Lambda$ by a  generalized Laplacian matrix.
\end{theorem}
\begin{proof}
Let $\Lambda$ be given as hypothesized.  We determine a matrix $M\in\mptn_L(K_n)$ satisfying $\spec(M) = \Lambda$ by induction on $n\geq 2$.  Using \cref{obs:k2}, we know $K_2$ realizes all possible $\Lambda$ with $|\Lambda| = 2$.  By induction, suppose $K_{n-1}$ realizes all such possible $\Lambda$ with $|\Lambda| = n - 1$.  For any $\Lambda$ with $|\Lambda| = n$, we will construct an $n \times n$ matrix $M\in\mptn_L(K_n)$ with $\spec(M) = \Lambda$.  

Let $\Lambda = \{0, \lambda_2, \ldots, \lambda_n\}$ be given.  We may assume that $0 < \lambda_2 \leq \cdots \leq \lambda_n$.  First we choose $\rho = \frac{\lambda_2}{n}$, $p = n - 1$, and $q = 1$ so that $\rho(p + q) = \lambda_2$.  Next, by the induction hypothesis, we choose a matrix $A\in\mptn_L(K_{n-1})$ with the nonzero eigenvalues $\lambda_3 - \rho, \ldots, \lambda_n - \rho$.  Note that these values are positive since $0<\rho = \frac{\lambda_2}{n} < \lambda_2$.  By \cref{lem:join}, the matrix  
\[
    M = \begin{bmatrix}
        A + \rho I & -\rho\bone \\
        -\rho\bone\trans & \rho(n - 1)
    \end{bmatrix}
\]
is a matrix in $\mptn_L(K_n)$ with spectrum $\Lambda$.  

Conversely, if a connected graph $G$ realizes all possible $\Lambda$, then it realizes the collection $\{0, \lambda, \ldots, \lambda\}$ by some generalized Laplacian matrix $L$.  Note that $L - \lambda I$ is a rank-one matrix in $\mptn(G)$.  Since $G$ is connected, $G$ must be the complete graph, as the only connected graph with minimum rank equal to one is the complete graph (see \cite[Observation 1.2]{FH}).
\end{proof}

We now consider the IEPL for stars, namely for the graph $K_{1,n-1}$.  Recall that the $k$-th elementary symmetric function $\sigma_k$ on a set $S$ is the sum of the product of all possible combinations of $k$ numbers in $S$, where we vacuously define $\sigma_0 = 1$ for any set.

\begin{theorem}
\label{thm:stars}
Let $\Lambda = \{0, \lambda_2, \ldots, \lambda_n\}$ be a set of real numbers with $0 < \lambda_2 \leq \cdots \leq \lambda_n$.  Then $\Lambda$ is Laplacian realizable for $K_{1,n-1}$ if and only if $(-1)^kf(-\lambda_k) \leq 0$ for $k = 2, \ldots, n-1$, where  
\[
    f(x) = \frac{\sigma_0}{1}x^{n-1} + \frac{\sigma_1}{2}x^{n-2} + \frac{\sigma_2}{3}x^{n-3} + \cdots + \frac{\sigma_{n-1}}{n}
\]
with $\sigma_k$ being the $k$-th elementary symmetric function of the positive numbers $\lambda_2, \ldots, \lambda_n$.  
\end{theorem}
\begin{proof}
Assume that the positive numbers $\lambda_2, \ldots, \lambda_n$ are given and that $\sigma_k$ is their $k$-th elementary symmetric function.  Thus, 
\[
    p(x) = \sigma_0x^{n-1} + \sigma_1x^{n-2} + \cdots + \sigma_{n-1}
\]
has $n - 1$ negative roots $-\lambda_2, \ldots, -\lambda_n$.  Consequently, 
\[
    q(x) = x^{n-1} p\left(\frac{1}{x}\right) = \sigma_{n-1}x^{n-1} + \sigma_{n-2}x^{n-2} + \cdots + \sigma_0
\]
also has $n - 1$ negative roots $-\frac{1}{\lambda_2}, \ldots, -\frac{1}{\lambda_n}$.  

On the other hand, suppose $M\in\mptn_L(K_{1,n-1})$ is defined by the weights $w_1,\ldots, w_{n-1}$ and has $\spec(M) = \{0,\lambda_2, \ldots, \lambda_n\}$.  By \cref{rem:NNW}, $M = NWN\trans$ and $W^{\frac{1}{2}}N\trans NW^{\frac{1}{2}}$ have the same nonzero eigenvalues $\{\lambda_2, \ldots, \lambda_n\}$.  By direct computation, $N\trans N = J + I$, where $J$ is the $(n-1)\times (n-1)$ all-ones matrix, and any $k\times k$ principal minor of $N\trans N$ is $k + 1$.  Therefore, given $\alpha \subseteq [n-1]$ with $|\alpha| = k$, the principal minor of $W^{\frac{1}{2}}N\trans NW^{\frac{1}{2}}$ is $(k+1)\prod_{i\in\alpha}w_i$.  Since the sum of all $k\times k$ principal minors of $W^{\frac{1}{2}}N\trans NW^{\frac{1}{2}}$ is $\sigma_k$ by its spectrum, we know the $k$-th elementary symmetric function of $w_1, \ldots, w_{n-1}$ is equal to $s_k = \frac{\sigma_k}{k+1}$ for $k = 0, \ldots, n-1$, which is the coefficient of $x^{n-1-k}$ in $f(x)$.  Therefore, $\Lambda$ is Laplacian realizable if and only if $f(x)$ has $n - 1$ negative roots, namely, $-w_1, \ldots, -w_{n-1}$.

Now we can relate $f(x)$ and $q(x)$ via the following.  
Define the function 
\[
    Q(x) = \int_0^x q(t)\,dt = \frac{\sigma_{n-1}}{n}x^{n} + \frac{\sigma_{n-2}}{n-1}x^{n-1} + \cdots + \frac{\sigma_0}{1}x.
\]
Note that the constant term is $Q(0) = \int_0^0q(t)\,dt = 0$.  Also, observe that $Q(x) = f(\frac{1}{x}) \cdot x^n$.  Therefore, $f(x)$ has $n - 1$ negative roots if and only if $Q(x)$ has $n - 1$ negative roots and a zero root.  
Since $Q'(x) = q(x)$, the critical points of $Q(x)$ occur at values $-\frac{1}{\lambda_2}, \ldots, -\frac{1}{\lambda_n}$.  Since $Q(x)$ has a positive leading coefficient, $Q(x)$ has $n - 1$ negative roots and a zero root if and only if $Q(-\frac{1}{\lambda_{n}}) < 0$ (which is guaranteed) and $Q(-\frac{1}{\lambda_{n-1}}) \geq 0$, $Q(-\frac{1}{\lambda_{n-2}}) \leq 0$, and so on.  Equivalently, $(-1)^{n-k} Q(-\frac{1}{\lambda_k}) \leq 0$ for $k = 2, \ldots, n-1$.  As $Q(-\frac{1}{\lambda_k})$ and $f(-\lambda_k)(-1)^n$ have the same sign, we know $(-1)^kf(-\lambda_k) \leq 0$ for $k = 2, \ldots, n-1$ if and only if $f(x)$ has $n - 1$ negative roots if and only if $\Lambda$ is Laplacian realizable for $K_{1,n-1}$.  
\end{proof}

We now consider two accompanying examples illustrating \cref{thm:stars}.

\begin{example}
For the star on $3$ vertices, $K_{1,2}$,
 \[
    f(-\lambda_2) = \frac{1}{6} \lambda_2(3\lambda_2 - \lambda_3).
\]
Thus, $\Lambda$ is Laplacian realizable if and only if $\lambda_3 \geq 3\lambda_2$.  This agrees with \cref{prop:p3}.
\end{example}

\begin{example}
Suppose $\{0,\lambda_2,\lambda_3,\lambda_4\}$ is Laplacian realizable for $K_{1,3}$.  By normalizing the trace, we may assume $\lambda_2 + \lambda_3 + \lambda_4 = 6$, which is the trace of the combinatorial Laplacian matrix of $K_{1,3}$. 
Thus, by replacing $\lambda_4 = 6 - \lambda_2 - \lambda_3$ we have 
\[
    \begin{aligned}
        f(-\lambda_2) &= \frac{1}{12}\lambda_2(-8\lambda_2^2 + \lambda_2\lambda_3 + \lambda_3^2 + 12\lambda_2 - 6\lambda_3), \text{ and }\\ 
        f(-\lambda_3) &= \frac{1}{12}\lambda_3(\lambda_2^2 + \lambda_2\lambda_3 - 8\lambda_3^2 - 6\lambda_2 + 12\lambda_3).
    \end{aligned}
\]
The boundaries of $f(-\lambda_2) \leq 0$ and $f(-\lambda_3) \geq 0$ are hyperbolas.  \cref{fig:g4} shows the simulation for possible pairs of $(\lambda_2,\lambda_3)$.  
\end{example}

\begin{remark}
The proof of \cref{thm:stars} also pointed out that $\Lambda = \{0,\lambda_2, \ldots, \lambda_n\}$ with $\lambda_2, \ldots, \lambda_n$ positive is Laplacian realizable for $K_{1,n-1}$ if and only if $f(x)$ has all roots real, which are necessarily all negative since the coefficients are positive.  For $n = 3$ and $n = 4$, the polynomial $f(x)$ has degree $2$ and $3$.  Then one may use the discriminant to determine if all roots are real or not.  For $n = 3$, the same conclusion $\lambda_3 \geq 3\lambda_2$ can be derived.  For $n = 4$ with the normalization $\lambda_2 + \lambda_3 + \lambda_4 = 6$, one may show that $\Lambda$ is Laplacian realizable for $K_{1,3}$ if and only if 
\[
    9s_2^2 - 4s_2^3 - 108s_3 - 27s_3^2 + 54s_2s_3 \geq 0,
\]
where $s_k = \frac{\sigma_k}{k+1}$ are the coefficients of $f(x)$.  
\end{remark}

\begin{figure}[h]
    \centering
    \includegraphics[width=\linewidth]{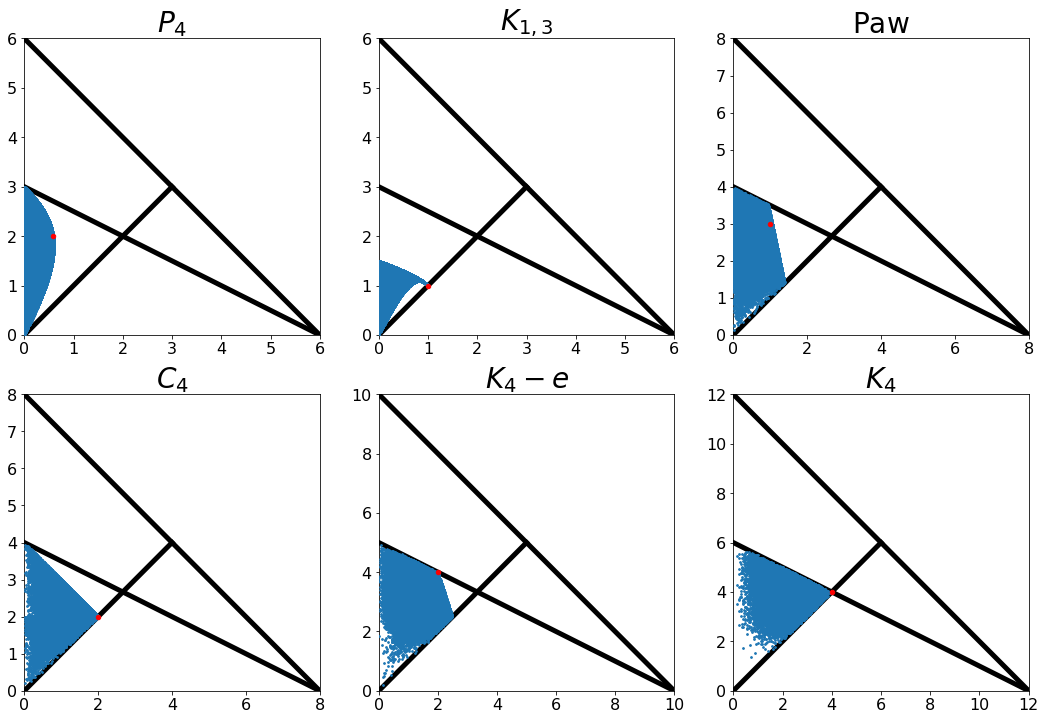}
    \caption{Data for graphs on $4$ vertices}
    \label{fig:g4}
\end{figure}

Although a complete characterization of the potential Laplacian spectra is presumed difficult, we can run a related simulation.  For each graph $G$ with $m$ edges, we treat each edge weight as a uniform random variable on $(0,1]$.  By assigning each edge weight randomly, we obtain a sample $A\in\mptn_L(G)$.  Since the eigenvalues of $kA$ is simply multiplying the eigenvalues of $A$ by $k$, we normalize $A$ by $\frac{2m}{\tr(A)}A$, mimicking the combinatorial Laplacian matrix.  Thus, each sample gives us a realizable spectrum of $G$, where the sum of eigenvalues is $2m$.  

\cref{fig:g4} demonstrates the simulation results for connected graphs on $4$ vertices, where the $x$-axis and the $y$-axis record the values of $\lambda_2$ and $\lambda_3$.  The red dot represent the spectrum of the combinatorial Laplacian matrix. 
For each graph $G$, we generate $\num[group-separator={,}]{100000}$ samples, compute their spectrum $\{0,\lambda_2, \lambda_3, \lambda_4\}$ with $\lambda_2 \leq \lambda_3 \leq \lambda_4$, and then plot their $(\lambda_2,\lambda_3)$ pairs.  In this way, $\lambda_4 = 2m - \lambda_2 - \lambda_3$ can be measured by the horizontal (or vertical) distance from the point to the straight line $x + y = 2m$. 
The other straight lines drawn on each picture are $x \leq y$ for the restriction $\lambda_2 \leq \lambda_3$, and $x + 2y \leq 2m$ for the restriction $\lambda_3 \leq \lambda_4 = 2m - \lambda_2 - \lambda_3$.

\section{The Ordered Multiplicity List Problem for Generalized Laplacian Matrices}

Let $G$ be a connected graph on $n$ vertices and suppose $L \in \mptn_L(G)$ with 
$\spec(L) = \{0^{(1)}, \lambda_2^{(m_2)}, \lambda_3^{(m_3)}, \ldots, \lambda_q^{(m_q)}\}$, where
$\lambda_i^{(m_i)}$ means the eigenvalue $\lambda_i$ of $L$ has multiplicity $m_i$ (in this case we have $1 + \sum m_i =n$). If further we assume that 
$0 < \lambda_2 < \lambda_3 < \cdots < \lambda_q$, then the {\em ordered multiplicity list for $L$} is defined to be $(m_1:=1,m_2, \ldots, m_q)$. A purpose of this section is to consider all possible ordered multiplicity lists over matrices in the set $\mptn_L(G)$. 

We begin this study of ordered multiplicity lists by considering the case of $n$ simple eigenvalues.

\begin{theorem}
\label{thm:discretespec}
The ordered multiplicity list $(1,\ldots,1)$ is Laplacian realizable for every connected graph.  
\end{theorem}
\begin{proof}
Let $G$ a connected graph on $n$ vertices.  Let $T$ be a spanning tree of $G$ with edges $e_1, \ldots, e_{n-1}$.  Let $G_k$ be the graph with the vertex set $V(G)$ and the edge set $\{e_1, \ldots, e_k\}$.  Then the number of components of $G_k$ is $n - k$.  

Choose $L_1$ to be the combinatorial Laplacian matrix of $G_1$. Hence $\spec(L_1) = \{0^{(n-1)}, 2\}$.  That is, the weight of $e_1$ is $1$.  

For $k = 2, \ldots, n-1$, assign the weights of $e_k$ using the following scheme.  
\begin{enumerate}
    \item Let $\delta$ be the minimum among all differences between any two distinct eigenvalues of $L_{k-1}$.  
    \item Let $L_k$ be the generalized Laplacian matrix of $G_k$ with the existing weights of $e_1, \ldots, e_{k-1}$ and a weight $w_k$ of $e_k$ to be determined.  
    \item Choose $w_k$ small enough so that the perturbation from $L_{k-1}$ to $L_k$ does not move the eigenvalues further than $\frac{\delta}{3}$.  
\end{enumerate}

Note that $L_k$ is obtained from $L_{k-1}$ by adding the $2 \times 2$ matrix  
\[
    w_k\begin{bmatrix}
        1 & -1 \\
        -1 & 1
    \end{bmatrix}
\]
to the $2 \times 2$ submatrix induced by the rows and columns corresponding to the endpoints of the edge $e_k$. Observe that the magnitude of the perturbation is continuous and controlled by $w_k$.  Moreover, by the choice of $w_k$, no distinct eigenvalues of $L_k$ will merge into a single eigenvalue in $L_{k+1}$.  However, one of the zero eigenvalues of $L_k$ becomes nonzero since $G_k$ has fewer components than $G_{k-1}$.  Continuing this process, $L_k$ has $k+1$ distinct eigenvalues.  In particular, $L_{n-1}$ has $n$ distinct eigenvalues.  

Finally, one may add small weights to edges in $E(G)\setminus E(T)$ in a way so that distinct eigenvalues remain distinct.  Thus, we have constructed a matrix in $\mptn_L(G)$ with the ordered multiplicity list $(1, \ldots, 1)$.  
\end{proof}

It is well known that for any tree $T$ and for any 
$A\in\mptn(T)$, the extreme eigenvalues of $A$ must be simple (see \cite{JSbook}). Further for any connected graph $G$ and any $L \in \mptn_L(G)$ we know that 0 is a simple eigenvalue. In the next result we verify that for a connected bipartite graph $G$ the largest eigenvalue of any  $L \in \mptn_L(G)$ is simple.

\begin{proposition}
\label{prop:bipspec}
Let $G$ be a connected bipartite graph.  Then any matrix in $\mptn_L(G)$ has its largest eigenvalue simple.  
\end{proposition}
\begin{proof}
Let $A$ be a matrix in $\mptn_L(G)$.  Since $G$ is a connected bipartite graph, there is a unique way to partition $V(G)$ into $X\dunion Y$ such that all edges are between $X$ and $Y$.  Let $D$ be a diagonal matrix whose rows/columns are indexed by $V(G)$.  Define the $i,i$-entry of $D$ as $1$ if $i\in X$ and $-1$ if $i\in Y$.  Thus, $D^{-1} = D$ and $D^{-1}AD$ is an irreducible nonnegative matrix with $\spec(D^{-1}AD) = \spec(A)$.  By the Perron--Frobenius theorem (see \cite{HJ1}), the largest eigenvalue of $D^{-1}AD$ is simple, and so is the largest eigenvalue of $A$.  
\end{proof}

As we have done in the previous section, we first summarize the potential ordered multiplicity lists for complete graphs and stars.  Though the complete spectra for paths remains open, its potential ordered multiplicity lists are resolved as observed below.

\begin{theorem}
\label{thm:knoml}
All possible ordered multiplicity lists among matrices in $\mptn_L(K_n)$ are $\{(m_1, \ldots, m_q): m_1 = 1,\ m_1 + \cdots + m_q = n\}$.  
\end{theorem}
\begin{proof}
This follows from \cref{thm:kn}.
\end{proof}

Even though \cref{thm:stars} provides a complete characterization of the possible Laplacian spectra for stars, it seems rather difficult to characterize the potential ordered multiplicity lists for generalized Laplacian matrices associated with stars. It turns out that we can determine the potential ordered multiplicity lists for stars directly.

\begin{theorem}
\label{thm:staroml} Suppose $n \geq 2$. Then all possible ordered multiplicity lists among matrices in $\mptn_L(K_{1,n-1})$ are those lists $(m_1, \ldots, m_q)$ such that  
\begin{itemize}
    \item $m_1 = m_q = 1$, 
    \item $m_k \geq 2 \implies m_{k+1} = 1$ for all $k$, and 
    \item $m_1 + \cdots + m_q = n$.  
\end{itemize}
\end{theorem}
\begin{proof}
Let $G = K_{1,n-1}$ and $(m_1, \ldots, m_q)$ the ordered multiplicity list of some matrix $A\in\mptn_L(G)$.  Then $m_1 = 1$ since $0$ is simple as $G$ is connected.  Also, $m_q = 1$ by \cref{prop:bipspec}.  Moreover, the Parter--Wiener theorem \cite{JSbook} forbids two   consecutive multiple eigenvalues, since there is only one such Parter vertex. To be more clear, let $\lambda_k$ and $\lambda_{k+1}$ be the $k$-th and the $(k+1)$-th smallest eigenvalues.  If $m_k$ and $m_{k+1}$ are both at least $2$, then the Parter--Wiener theorem guarantees a vertex $v$ of degree at least three, which must be the central vertex in the case of stars, such that the $\mult_{A(v)}(\lambda_k) = m_k + 1$ and $\mult_{A(v)}(\lambda_{k+1}) = m_{k+1} + 1$.  However, this is impossible according to the Cauchy interlacing theorem.

On the other hand, let $\oml = (m_1, \ldots, m_q)$ be a list with $m_1 = m_q = 1$, no consecutive multiple eigenvalues, and $m_1 + \cdots + m_q = n$.  Let $1$ be the central vertex and 
\[
    A = \begin{bmatrix}
        w_1 + \cdots + w_{n-1} & -w_1 & \cdots & -w_{n-1} \\
        -w_1 & w_1 & ~ & ~ \\
        \vdots & ~ & \ddots & ~ \\
        -w_{n-1} & ~ & ~ & w_{n-1}
    \end{bmatrix}
    \in\mptn_L(K_{1,n-1}).
\]
Obtain a new list $\oml'$ from $\oml$ by removing $m_1$ and replacing every $m_k, m_{k+1}$ with $m_k + 1$ for any $m_k \geq 2$.  Note that necessarily $m_{k+1} = 1$, so the sum of $\oml'$ is $n - 1$.  
Now we choose $\{w_1, \ldots, w_{n-1}\}$ as a multi-set with the ordered multiplicity list $\oml'$.  

We claim that $A$ has the ordered multiplicity list $\oml$.  Observe that the $\spec(A(1)) = \{w_1, \ldots, w_{n-1}\}$.  Also, $A(1) - w_k I$ contains some zero rows for each $k$, so $A(1,1]$ is not in the column space of $A(1) - w_k I$.  Thus, 
\[
    \mult_A(w_k) = \nul(A - w_k I) = \nul(A(1) - w_k I) - 1 = \mult_{A(1)}(w_k) - 1,
\]
so from $A(1)$ to $A$, every multiplicity drops by $1$, and thus the simple eigenvalues disappear.  However, by the Cauchy interlacing theorem, there is a new eigenvalue in each open interval between any consecutive distinct eigenvalues of $A(1)$.  Therefore, the ordered multiplicity list of $A$ is $\oml$.  
\end{proof}

For the path, $P_n$, the allowable ordered multiplicity list problem is straightforward.
For a given graph $G$, we define the maximum multiplicity of $G$, denoted by $M(G)$, as
\[ M(G) = \max\{\mult_A(\lambda): \lambda \in \spec(A),\ A \in \mptn(G)\}.\]

\begin{theorem}
The only possible ordered multiplicity list among matrices in $\mptn_L(P_n)$ is $\{(1,\ldots,1)\}$.  
\end{theorem}
\begin{proof}
This follows from \cref{thm:discretespec} and the fact that $M(P_n) = 1$ (see \cite{FH}).  
\end{proof}

We move forward by investigating the inverse ordered multiplicity list problem for generalized Laplacians on graphs with a small number of vertices.  Note that every graph of order at most $3$ is either a complete graph or path, so we focus on connected graphs with four vertices.  Below is a summary containing all connected graphs on four vertices and the allowed multiplicity lists.

\begin{enumerate}
\item $P_4$: $(1,1,1,1)$
\item $K_{1,3}$: $(1,1,1,1)$, $(1,2,1)$. 
\item $\Paw$: $(1,1,1,1)$, $(1,1,2)$, $(1,2,1)$
\item $C_4$: $(1,1,1,1)$, $(1,2,1)$
\item $K_4 - e$: $(1,1,1,1)$, $(1,1,2)$, $(1,2,1)$
\item $K_4$: $(1,1,1,1)$, $(1,1,2)$, $(1,2,1)$, $(1,3)$
\end{enumerate}

As the cases of complete graphs, stars, and paths are already complete by previous considerations, we only verify the results for graphs: $\Paw$, $C_4$, and $K_4 - e$. 
We begin with the $\Paw$ graph.

\begin{theorem}
All possible ordered multiplicity lists among matrices in $\mptn_L(\Paw)$ are $\{(1,1,1,1), (1,1,2), (1,2,1)\}$.
\end{theorem}
\begin{proof}
Since $M(\Paw) = 2$ (see \cite{FH}) and $0$ is always simple as the smallest eigenvalue, it follows that the only possible ordered multiplicity lists are$ (1,1,1,1)$, $(1,1,2)$, and $(1,2,1)$.  We know $(1,1,1,1)$ is realizable by \cref{thm:discretespec}.

Consider the matrix  
\[
    \begin{bmatrix}
        5 & -3 & 0 & -2 \\
        -3 & 5 & 0 & -2 \\
        0 & 0 & 2 & -2 \\
        -2 & -2 & -2 & 6
    \end{bmatrix} \text{ and }
    \begin{bmatrix}
        4 & -1 & 0 & -3 \\
        -1 & 4 & 0 & -3 \\
        0 & 0 & 10 & -10 \\
        -3 & -3 & -10 & 16
    \end{bmatrix}
\]
in $\mptn_L(\Paw)$.  Note that these matrices has the spectra $\{0,2,8,8\}$ and $\{0,5,5,24\}$, respectively.  Therefore, the ordered multiplicity list $(1,1,2)$ and $(1,2,1)$ are realizable.   
\end{proof}

We now consider the allowed multiplicity lists associated with the 4-cycle.

\begin{theorem}
The possible ordered multiplicity lists among matrices in $\mptn_L(C_4)$ are $\{(1,1,1,1), (1,2,1)\}$.
\end{theorem}
\begin{proof}
Since $M(C_n) = 2$ (see \cite{FH}) and $0$ is always simple as the smallest eigenvalue, the possible ordered multiplicity lists are$ (1,1,1,1)$, $(1,1,2)$, and $(1,2,1)$.  However, $(1,1,2)$ is not possible since $C_n$ is a bipartite graph by \cref{prop:bipspec}.

We know $(1,1,1,1)$ is realizable by \cref{thm:discretespec}.  On the other hand, the combinatorial Laplacian matrix of $C_4$ has spectrum $\{0,2,2,4\}$, so the ordered multiplicity list $(1,2,1)$ is realizable.
\end{proof}

Finally, we consider the complete graph on four vertices with one edge deleted.

\begin{theorem}
All possible ordered multiplicity lists among matrices in $\mptn_L(K_4 - e)$ are $\{(1,1,1,1), (1,1,2), (1,2,1)\}$.
\end{theorem}
\begin{proof}
Since $M(K_4 - e) = 2$ (see \cite{FH}) and $0$ is always simple as the smallest eigenvalue, the only possible ordered multiplicity lists are$ (1,1,1,1)$, $(1,1,2)$, and $(1,2,1)$.  We know $(1,1,1,1)$ is realizable by \cref{thm:discretespec}.

The matrix  
\[
    \begin{bmatrix}
        4 & 0 & -3 & -1 \\
        0 & 4 & -3 & -1 \\
        -3 & -3 & 7 & -1 \\
        -1 & -1 & -1 & 3
    \end{bmatrix}
\]
is in $\mptn_L(K_4)$ and has the spectrum $\{0,4,4,10\}$, so the ordered multiplicity list $(1,2,1)$ is realizable. 

On the other hand, the combinatorial Laplacian matrix of $K_4 - e$ has spectrum $\{0,2,4,4\}$, so the ordered multiplicity list $(1,1,2)$ is realizable.
\end{proof}

\subsection{Case of Three Distinct Eigenvalues}
\label{subs:3evals}

As we have seen in the previous sections, a complete characterization of the potential spectra (along with specified multiplicities) is rather difficult even for small graphs such as $P_4$ and $C_4$.  However, we are able to provide partial solutions for these cases when there are only three distinct eigenvalues.  We begin with the stars and then discuss the graphs on $4$ vertices.  

We determine that for stars, surprisingly, there is only one matrix, up to rescaling, that can achieve the ordered multiplicity list $(1,n-2,1)$.  

\begin{theorem}
For $n \geq 4$, the only matrix in $\mptn_L(K_{1,n-1})$ that achieves the ordered multiplicity list $(1,n-2,1)$, is $kL$ with $k > 0$ and $L$ being the combinatorial Laplacian matrix for $K_{1,n-1}$  In this case, the corresponding spectrum is given by $\{0, k^{(n-2)}, kn\}$.
\end{theorem}
\begin{proof}
Let $1$ be the central vertex of $K_{1,n-1}$ and $A\in\mptn_L(K_{1,n-1})$.  Suppose $\spec(A) = \{0, \lambda_2^{(n-2)}, \lambda_n\}$ with $0 < \lambda_2 < \lambda_n$.  By the Parter--Wiener theorem \cite{JSbook}, there is a vertex $v$ of degree at least three such that $\mult_{A(v)}(\lambda_2) = \mult_A(\lambda_2) + 1$.  Since $1$ is the only vertex of degree at least three, necessarily $v = 1$.   Since $A(1)$ is an $(n-1)\times (n-1)$ diagonal matrix with spectrum $\{\lambda_2^{(n-1)}\}$, it follows that $A(1) = \lambda_2 I_{n-1}$, which then determines all the weights on $K_{1,n-1}$.  That is, $A = kL$ with $k = \lambda_2 > 0$ and $L$ the combinatorial Laplacian matrix.  By direct computation, the corresponding spectrum of $A$ is then $\{0, k^{(n-2)}, kn\}$.
\end{proof}

We will use the following proposition to handle the remaining cases with $n=4$ and two simple eigenvalues. Note that the statement of the proposition does not require $\lambda < \mu$. 

\begin{proposition}
\label{prop:3evals}
Let $G$ be a graph on $n$ vertices.  Then a matrix $A\in\mptn_L(G)$ exists with three distinct eigenvalues and $\spec(A) = \{0, \lambda^{(n-2)}, \mu\}$  with $g = \frac{\mu - \lambda}{\lambda}$ if and only if there is a vector $\bu = \begin{bmatrix} u_i \end{bmatrix}$ such that  
\begin{itemize}
    \item $\|\bu\|^2 = 1$, 
    \item $\bone\trans\bu = 0$, 
    \item $gu_iu_j = \frac{1}{n}$ if $\{i,j\}\in E(\overline{G})$, and 
    \item $gu_iu_j < \frac{1}{n}$ if $\{i,j\}\in E(G)$. 
\end{itemize}
\end{proposition}
\begin{proof}
Suppose $A\in\mptn_L(G)$ has three distinct eigenvalues and has corresponding $\spec(A) = \{0, \lambda^{(n-2)}, \mu\}$ with $g = \frac{\mu - \lambda}{\lambda}$.  Then $\frac{1}{\lambda}A - I$ has two nonzero eigenvalues $-1$ and $g$, so its spectral decomposition is  
\[
    \frac{1}{\lambda}A - I = (-1)\left(\frac{1}{n}J\right) + 0P_\lambda + g\bu\bu\trans,
\]
where $J$ is the all-ones matrix, $P_\lambda$ is the projection matrix onto the eigenspace of $A$ with respect to $\lambda$, and $\bu$ is a unit eigenvector of $A$ with respect to $\mu$.  Thus, we found a vector $\bu$ such that $\|\bu\|^2 = 1$, $\bone\trans\bu = 0$, and 
\[
    I - \frac{1}{n}J + g\bu\bu\trans = 
    \frac{1}{\lambda}A\in\mptn_L(G).
\]
By the definition of $\mptn_L(G)$, we have $gu_iu_j = \frac{1}{n}$ if $\{i,j\} \in E(\overline{G})$ and $gu_iu_j < \frac{1}{n}$ if $\{i,j\} \in E(G)$.  

For the converse, let $g \neq 0$ and a vector $\bu$ satisfying the four above conditions be given. Construct a matrix  $B$ given by
\[
    B = I - \frac{1}{n}J + g\bu\bu\trans.
\]
Since $\bone\trans\bu = 0$, we have $B\bone = \bone - \bone = \bzero$.  Together with the third and the fourth conditions, this implies $B\in\mptn_L(G)$.  Since $\|\bu\|^2 = \bu\trans\bu = 1$, we also have $B\bu = \bu + g\bu = (1 + g)\bu$.  Observe that for any vector $\bv$ that is orthogonal to $\bone$ and $\bu$, we have $B\bv = \bv$.  Thus, $\spec(B) = \{0, 1^{(n-2)}, 1 + g\}$.  For any desired $\lambda$ and $\mu$ with $\frac{\mu - \lambda}{\lambda} = g$, the matrix $A = \lambda B$ has $\spec(A) = \{0, \lambda^{(n-2)}, \mu\}$.  
\end{proof}

Since $A\in\mptn_L(G)$ if and only if $\frac{1}{\lambda}A\in\mptn_L(G)$ for any $\lambda > 0$.  In the following, we customize the conditions in \cref{prop:3evals} into the case when $n = 4$ and $\lambda = 1$ as follows.  We first note that with $\bu = (x,y,z,w)\trans$ the conditions $\|\bu\|^2 = 1$ and $\bone\trans\bu = 0$ can be written as  
\begin{equation}
\label{eq:4u}
    \begin{aligned}
        x^2 + y^2 + z^2 + w^2 &= 1, \\
        x + y + z + w &= 0.
    \end{aligned}
\end{equation}


\begin{remark}
\label{rem:4u}
Let $G$ be a graph on $4$ vertices.  Then a matrix $A\in\mptn_L(G)$ exists with $\spec(A) = \{0, 1^{(2)}, \mu\}$ such that $0 < 1 < \mu$ if and only if there is a vector $\bu = (x,y,z,w)\trans$ such that  
\begin{itemize}
    \item \cref{eq:4u} holds, and 
    \item among all off-diagonal entries of $\bu\bu\trans$, the \emph{maximum} is \emph{positive} and occurs precisely on those entries corresponding to non-edges
    (and the maximum value is $\frac{1}{4g} = \frac{1}{4(\mu - 1)}$).  
\end{itemize}

On the other hand, a matrix $A\in\mptn_L(G)$ exists with $\spec(A) = \{0, 1^{(2)}, \mu\}$ such that $0 < \mu < 1$ if and only if there is a vector $\bu = (x,y,z,w)\trans$ such that  
\begin{itemize}
    \item \cref{eq:4u} holds, and 
    \item among all off-diagonal entries of $\bu\bu\trans$, the \emph{minimum} is \emph{negative} and occurs precisely on those entries corresponding to non-edges (and the minimum value is $\frac{1}{4g} = \frac{1}{4(\mu - 1)}$).
\end{itemize}
\end{remark}

\begin{observation}
\label{obs:agmg}
Given $\alpha, \beta > 0$, the system of equations  
\[\begin{aligned}
    p^2 + q^2 &= \alpha, \\
    pq &= \beta,
\end{aligned}\]
has a solution with $p,q > 0$ if and only if $\alpha \geq 2\beta$.  
\end{observation}

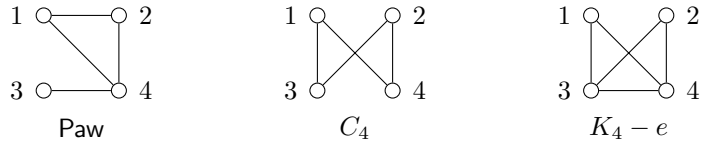
\begin{figure}[h]
\begin{center}
\begin{tikzpicture}
\draw[opacity=0] (0.5,-1) -- (0.5, 1.5);
\node[label={left:$1$}] (1) at (0,1) {};
\node[label={right:$2$}] (2) at (1,1) {};
\node[label={left:$3$}] (3) at (0,0) {};
\node[label={right:$4$}] (4) at (1,0) {};
\draw (4) -- (1) -- (2) -- (4) -- (3);
\node[draw=none, rectangle] at (0.5, -0.5) {$\Paw$};
\end{tikzpicture}
\hfil
\begin{tikzpicture}
\draw[opacity=0] (0.5,-1) -- (0.5, 1.5);
\node[label={left:$1$}] (1) at (0,1) {};
\node[label={right:$2$}] (2) at (1,1) {};
\node[label={left:$3$}] (3) at (0,0) {};
\node[label={right:$4$}] (4) at (1,0) {};
\draw (3) -- (1) -- (4) -- (2) -- (3);
\node[draw=none, rectangle] at (0.5, -0.5) {$C_4$};
\end{tikzpicture}
\hfil
\begin{tikzpicture}
\draw[opacity=0] (0.5,-1) -- (0.5, 1.5);
\node[label={left:$1$}] (1) at (0,1) {};
\node[label={right:$2$}] (2) at (1,1) {};
\node[label={left:$3$}] (3) at (0,0) {};
\node[label={right:$4$}] (4) at (1,0) {};
\draw (3) -- (1) -- (4) -- (2) -- (3) -- (4);
\node[draw=none, rectangle] at (0.5, -0.5) {$K_4 - e$};
\end{tikzpicture}
\end{center}
\caption{Labeled graphs for Theorems~\ref{thm:paw3}, \ref{thm:c43}, and \ref{thm:k4e3}.}
\label{fig:labeledg}
\end{figure}

Following the order previously, we begin with the $\Paw$ graph.

\begin{theorem}
\label{thm:paw3}
The multiset $\{0, \lambda^{(2)}, \mu\}$ is Laplacian realizable for the $\Paw$ graph if and only if $\mu \geq (2 + \sqrt{3})\lambda$ or $0 < \mu \leq (2 - \sqrt{3})\lambda$.
\end{theorem}
\begin{proof}
Observe that, under the assumption $x = y$, \cref{eq:4u} is equivalent to  
\begin{equation}
\label{eq:4paw}
    \begin{aligned}
        6x^2 + 2z^2 &= 1 - 4xz, \\
        x + y + z + w &= 0,
    \end{aligned}
\end{equation}
by substituting $w^2 = (x + y + z)^2$.

We adopt the notation from \cref{prop:3evals} and the subsequent discussion afterwards.  Since $g = \frac{\mu - \lambda}{\lambda}$, it is sufficient to verify that the multi-set is Laplacian realizable if and only if $g \geq 1 + \sqrt{3}$ or $-1 < g \leq 1 - \sqrt{3}$.  Let $\Paw$ be labeled as in \cref{fig:labeledg}. 

We focus on the case of $g > 0$ first.  Suppose a vector $\bu$ satisfying the conditions in \cref{rem:4u} exists.  By replacing $\bu$ with $-\bu$ if necessary, we may assume $x \geq 0$.    Thus, we have $xz = yz = \frac{1}{4g} > 0$.  Hence $x = y$ and $x,y,z$ are necessarily positive.  Therefore, \cref{eq:4u} and \cref{eq:4paw} are equivalent.  Since $xz = \frac{1}{4g}$, we have the equations  
\[
    \begin{aligned}
        6x^2 + 2z^2 &= 1 - \frac{1}{g}, \\
        (\sqrt{6}x)(\sqrt{2}z) &= \frac{\sqrt{3}}{2g}.
    \end{aligned}
\]
with $\sqrt{6}x, \sqrt{2}z > 0$.  By \cref{obs:agmg}, the equations have positive solutions if and only if  
\[
    1 - \frac{1}{g} \geq 2\cdot \frac{\sqrt{3}}{2g},
\]
or equivalently, $g \geq 1 + \sqrt{3}$.

Conversely, if $g \geq 1 + \sqrt{3}$, then the above equations has a positive solution for $\sqrt{6}x$ and $\sqrt{2}z$.  By symmetry, we may assume that $\sqrt{6}x \leq \sqrt{2}z$, which further implies $x < z$.  Then we may assign $y = x$ and $w = -(x + y + z)$.  Thus, \cref{eq:4paw} holds.  Moreover, $xw, yw, zw < 0 < xz$ and $xy < xz$ by our choice of $x$ and $z$.  Thus, the desired $\bu$ exists.

Now we consider the other case $g < 0$.  Suppose $\bu$ exists.  Then we have $x = y > 0$ and $z < 0$ since $xz = yz = \frac{1}{4g} < 0$.  Again, \cref{eq:4u} and \cref{eq:4paw} are equivalent, which implies the equations   
\[
    \begin{aligned}
        6x^2 + 2z^2 &= 1 - \frac{1}{g}, \\
        (\sqrt{6}x)(-\sqrt{2}z) &= -\frac{\sqrt{3}}{2g}.
    \end{aligned}
\]
The equations have positive solutions if and only if 
\[
    1 - \frac{1}{g} \geq 2\cdot -\frac{\sqrt{3}}{2g}.
\]
Note that multiplying by $g < 0$ on both sides flips the inequality, so the above inequality is equivalent to $g \leq 1 - \sqrt{3}$.  The condition $-1 < g$ follows from the fact $\mu > 0$.

Conversely, if $-1 < g \leq 1 - \sqrt{3}$, then we may find a positive solution for $\sqrt{6}x$ and $-\sqrt{2}z$.  By symmetry, we choose the solution so that $\sqrt{6}x \leq -\sqrt{2}z$, which further implies $x < -z$.  Then we may assign $y = x$ and $w = -(x + y + z)$.  Thus, \cref{eq:4paw} holds.  We also have $xy > 0 > xz$.  When $g < 0$, we may compute  
\[
    \begin{aligned}
        w^2 &= (x + y + z)^2 = (2x + z)^2 \\
        &= 4x^2 + z^2 + 4xz = x^2 + (3x^2 + z^2) + 4xz \\
        &= x^2 + \frac{1}{2}\left(1 - \frac{1}{g}\right) + \frac{1}{g} = x^2 + \frac{1}{2}\left(1 + \frac{1}{g}\right).
    \end{aligned}
\]
As $\frac{1}{1 - \sqrt{3}} \leq \frac{1}{g} < -1$, we have $|w| < x < -z$, which implies $xw, yw, zw > xz$.  Thus, the desired $\bu$ exists.
\end{proof}

We now consider the four cycle.

\begin{theorem}
\label{thm:c43}
The multiset $\{0, \lambda^{(2)}, \mu\}$ is Laplacian realizable for $C_4$ if and only if $\mu \geq 2\lambda$.
\end{theorem}
\begin{proof}
We observe that, under the assumption $xy = zw$, \cref{eq:4u} is equivalent to  
\begin{equation}
\label{eq:4c4}
    \begin{aligned}
        2x^2 + 2y^2 &= 1, \\
        x + y + z + w &= 0,
    \end{aligned}
\end{equation}
by substituting 
\[
    z^2 + w^2 = (z + w)^2 - 2zw = (x + y)^2 - 2zw = (x + y)^2 - 2xy = x^2 + y^2.
\]

The statement is equivalent to $g \geq 1$. Applying  \cref{prop:bipspec} we note that both $g < 0$ and $\mu < \lambda$ are impossible, so we only consider the case of $g > 0$.  Let $C_4$ be labeled as in \cref{fig:labeledg}.  

Suppose a vector $\bu$ satisfying the conditions exists.  Then $xy = zw = \frac{1}{4g} > 0$.  We may assume $x,y > 0$ and $z, w < 0$, since their sum is zero and they cannot be all positive.  Thus, \cref{eq:4u} and \cref{eq:4c4} are equivalent.  Therefore we have the equations  
\[
    \begin{aligned}
        x^2 + y^2 &= \frac{1}{2}, \\
        xy &= \frac{1}{4g}.
    \end{aligned}
\]
By \cref{obs:agmg}, the equations have a positive solution if and only if $\frac{1}{2} \geq 2\cdot\frac{1}{4g}$, which is equivalent to $g \geq 1$.  

Conversely, if $g \geq 1$, then a positive solution for $x,y$ exists.  By assigning $z = -x$ and $w = -y$, the vector satisfies \cref{eq:4c4}.  Also, $xz,xw,yz,yw < 0 < xy$, confirming the desired vector $\bu$ exists.
\end{proof}

Finally we consider the complete graph on four vertices with an edge removed. 

\begin{theorem}
\label{thm:k4e3}
The multiset $\{0, \lambda^{(2)}, \mu\}$ is Laplacian realizable for $K_4 - e$ if and only if $\mu > 2\lambda$ or $0 < \mu \leq \frac{1}{2}\lambda$.
\end{theorem}
\begin{proof}
We observe that, by setting $\epsilon = xy - zw$, \cref{eq:4u} is equivalent to  
\begin{equation}
\label{eq:4k4e}
    \begin{aligned}
        2x^2 + 2y^2 &= 1 - 2\epsilon, \\
        2z^2 + 2w^2 &= 1 + 2\epsilon, \\
        x + y + z + w &= 0,
    \end{aligned}
\end{equation}
by substituting 
\[
    z^2 + w^2 = (z + w)^2 - 2zw = (x + y)^2 - 2zw = x^2 + y^2 + 2\epsilon.
\]

The statement is equivalent to $g \geq 1$ or $-1 < g \leq -\frac{1}{2}$.  We label the vertices of $K_4 - e$ as in \cref{fig:labeledg}.

We first consider $g > 0$.  Then, we have $xy = \frac{1}{4g} > 0$ and $xy > zw$, which means $\epsilon > 0$ in \cref{eq:4k4e}.  We may assume $x, y > 0$.  Hence
\[
    \begin{aligned}
        x^2 + y^2 &= \frac{1}{2} - \epsilon, \\
        xy &= \frac{1}{4g},    
    \end{aligned}
\]
which has a positive solution only when  
\[\frac{1}{2} > \frac{1}{2} - \epsilon \geq 2\cdot\frac{1}{4g} = \frac{1}{2g},\]
or equivalently  $g > 1$.  

Conversely, if $g > 1$, then we choose a small $\epsilon$ so that $0 < \epsilon \leq \frac{1}{2} - \frac{1}{2g}$ and $\epsilon < \frac{1}{4g}$.  Thus, the above equations have a positive solution for $x, y$.  With $\epsilon$ chosen, we then solve  
\[
    \begin{aligned}
        (-z)^2 + (-w)^2 &= \frac{1}{2} + \epsilon, \\
        (-z)(-w) &= \frac{1}{4g} - \epsilon.
    \end{aligned}
\]
We may double check that $\frac{1}{2} + \epsilon$ and $\frac{1}{4g} - \epsilon$ are positive by our choice of $\epsilon$ and $g$.  Moreover, 
\[
    \frac{1}{2} + \epsilon > \frac{1}{2} > 2\cdot \frac{1}{4g} > 2\left(\frac{1}{4g} - \epsilon\right).
\]
Therefore, we arrive at a positive solution for $-z$ and $-w$.  Note that this choice yields
\[
    (x + y)^2 = x^2 + y^2 + 2xy = \frac{1}{2} - \epsilon + \frac{1}{2g} = z^2 + w^2 + 2zw = (z + w)^2,
\]
so we have $x + y + z + w = 0$.  Thus, \cref{eq:4k4e} holds.  With $xz, xw, yz, yw < 0 < xy$ and by our choice of $zw < xy$, the desired $\bu$ is found.

Next, consider the case $g < 0$.  Suppose $\bu$ exists.  Then we have $xy = \frac{1}{4g} < 0$ and $xy < zw$, which means $\epsilon < 0$ in \cref{eq:4k4e}.  We may assume $x > 0$ and $y < 0$.  Now  
\[
    \begin{aligned}
        x^2 + y^2 &= \frac{1}{2} - \epsilon, \\
        x(-y) &= -\frac{1}{4g}.    
    \end{aligned}
\]
Note that $z^2 + w^2 = \frac{1}{2} + \epsilon$ indicates that $\epsilon \geq -\frac{1}{2}$.  Therefore, the above equations have a solution only when 
\[
    1 \geq \frac{1}{2} - \epsilon \geq -2 \cdot \frac{1}{4g},
\]
which is equivalent to $g \leq -\frac{1}{2}$ since multiplying $g < 0$ on both sides flips the the inequality.  The condition $-1 < g$ is guaranteed by $0 < \mu$. 

Conversely, if $-1 < g \leq -\frac{1}{2}$, then the above equations have a positive solution for $x$ and $-y$.  By symmetry, we may assume $x \geq -y$ so that $x + y \geq 0$.  We then choose $\epsilon = \frac{1}{6g} - \frac{1}{6} < 0$ and solve 
\[
    \begin{aligned}
        z^2 + w^2 &= \frac{1}{2} + \epsilon = \frac{1}{6g} + \frac{1}{3}, \\
        zw &= \frac{1}{4g} - \epsilon = \frac{1}{12g} + \frac{1}{6}
    \end{aligned}
\]
to produce a solution $z = w = \sqrt{\frac{1}{12g} + \frac{1}{6}} \geq 0$.  Thus, 
\[
    (x + y)^2 = x^2 + y^2 + 2xy = \frac{1}{2} - \epsilon + \frac{1}{2g} = z^2 + w^2 + 2zw = (z + w)^2,
\]
giving $x + y + z + w = 0$, since $x + y \geq 0$ and $z + w \geq 0$.  Moreover, $xz, xw, zw \geq 0 > xy$ and $yz, yw > xy$ by our choice of $x \geq -y$ and the fact $x^2 + y^2 > z^2 + w^2$.  Consequently, the desired $\bu$ exists.
\end{proof}

Note that the cases for $K_4$ are not included here since the inverse eigenvalue problem for any complete graph was addressed in the previous section.

\section{Minimum Variance}
\label{sec:var}

As we have seen in previous sections, there is a lower bound on the ratio $\frac{\lambda_2(A)}{\lambda_3(A)}$ for any matrices $A\in\mptn_L(P_3)$, and, in general, the nonzero eigenvalues cannot be too concentrated unless $G$ is a complete graph.  In this section, we consider the minimum variance of the nonzero eigenvalues when we fix the sum of weights to be equal to the number of edges.  Equivalently, we focus on {\em size-normalized generalized} Laplacian matrices  defined as
\[
    \mptn^s_L(G) = \{A \in\mptn_L(G): \tr(A) = 2|E(G)|\}.
\]

\begin{definition}
Let $A\in\mptn_L(G)$ with the spectrum $\{0, \lambda_2, \ldots, \lambda_n\}$.  Define  
\[
    \mathbb{E}(A) = \frac{\sum_{i=2}^n \lambda_i}{n-1},\ 
    \Var(A) = 
    \frac{\sum_{i=2}^n (\lambda_i - \mathbb{E}(A))^2}{n-1},\text{ and }
    p_2(A) = \sum_{i=2}^n \lambda_i^2.
\]
\end{definition}

\begin{definition}
The \emph{minimum variance} of $G$ is defined as 
\[
    \mv(G) = 
    \inf\{ \Var(A) : A\in\mptn_L^s(G) \}.
\]
\end{definition}

\begin{remark}
\label{rem:varformula}
It is well-known (and by direct computation as well) that  
\[
    \Var(A) = \frac{p_2(A)}{n-1} - \mathbb{E}(A)^2.
\]
Moreover, when $G$ is a graph on $m$ edges and $A\in\mptn_L^s(G)$, we have 
\[
    \Var(A) = \frac{\tr(A^2)}{n-1} - \left(\frac{2m}{n-1}\right)^2.
\]
\end{remark}

Since $\mptn_L^s(G)$ is not a compact set, there might not be a matrix $A\in\mptn_L^s(G)$ that attains $\mv(G)$.  However, the infimum must occur in the closure of $\mptn_L^s(G)$, which allows some of the weights to be zero.  This means the variances realized by matrices in $\mptn_L(G)$ can be arbitrarily close to $\mv(G)$, which still provide significant worthwhile information for studying the IEPL.  

On the other hand, let $\onemv(G)$ be the variance of the combinatorial Laplacian matrix of $G$.  Then by definition  
\[\mv(G) \leq \onemv(G).\]
Here we provide alternative methods to calculate $\Var(A)$.  

Given a graph $G=(V,E)$, the {\em line graph} of $G$, denoted by $L(G)$, is graph whose vertices are $E$ and two edges $e$ and $e'$ are adjacent in $L(G)$ if they share a common vertex.

\begin{proposition}
\label{prop:m2formula}
Let $G$ be a connected graph and $A\in\mptn_L(G)$ with the weight vector $\bw$.  Then $p_2(A) = \bw\trans M_2\bw$, where $M_2 = 4I + B$ and $B$ is the adjacency matrix of the line graph of $G$. 
\end{proposition}
\begin{proof}
Let $\bw = \begin{bmatrix} w_i \end{bmatrix}$, $W = \diag(\bw)$, and $W^{\frac{1}{2}}$ the diagonal matrix whose diagonal entries are given by $w_i^{\frac{1}{2}}$, for $i=1,2,\ldots, m$.  Let $N$ be an vertex-edge incidence matrix of $G$ and $M = \begin{bmatrix} m_{i,j} \end{bmatrix} = 2I + B$.  Then $A = NWN\trans$ and $M = N\trans N$.  Moreover, $A = NW^{\frac{1}{2}}W^{\frac{1}{2}}N\trans$ and $W^{\frac{1}{2}}MW^{\frac{1}{2}} = W^{\frac{1}{2}}N\trans NW^{\frac{1}{1}}$ have the same nonzero eigenvalues.  

Therefore, 
\[
    \begin{aligned}
    p_2(A) &= \tr((W^{\frac{1}{2}}MW^{\frac{1}{2}})^2) \\ 
    &= \sum_{i,j} (w_i^\frac{1}{2} m_{i,j}w_j^\frac{1}{2})^2 = \sum_{i,j} w_im_{i,j}^2w_j \\
    &= \bw\trans (M\circ M)\bw,
    \end{aligned}
\]
where $\circ$ signifies the entrywise product of conformally sized matrices.  We also note that $M\circ M = 4I + B = M_2$.  
\end{proof}

Observe that $N\trans N$ is a positive semidefinite matrix.  (Or, equivalently, the adjacency matrix of a line graph has its minimum eigenvalue at least $-2$; see, e.g., \cite{BapatGM14}.)  Therefore, $M_2$ is invertible with its minimum eigenvalue at least $2$.  

By \cref{rem:varformula} and \cref{prop:m2formula}, for any given graph $G$ with $m$ edges, finding the value of $\mv(G)$ is equivalent to finding the minimum of $\bw\trans M_2\bw$ subject to $\bone\trans \bw = m$ and $\bw$ being entrywise nonnegative.  

By dropping the last condition, this minimization problem has a unique solution.  

\begin{proposition}
\label{prop:mvsol}
Let $G$ be a graph with $m$ edges and $M_2 = 4I + B$, where $B$ is the adjacency matrix of the line graph of $G$.  The minimization problem
\[
    \begin{array}{ll}
        \min & \bw\trans M_2 \bw \\
        \text{\rm subject to} &  \bone\trans\bw = m
    \end{array}
\]
attains its minimum at 
\[
    \bw = kM_2^{-1}\bone \text{ with }k = \frac{m}{\bone\trans M_2^{-1}\bone}
\]
and the minimum value is  
\[
    \bw\trans M_2\bw = \frac{m^2}{\bone\trans M_2^{-1}\bone}.
\]
Therefore, 
\[
    \mv(G) \geq
    \frac{m^2}{(n-1)\bone\trans M_2^{-1}\bone} - \left(\frac{2m}{n-1}\right)^2.
\]
while the equality holds if $M_2^{-1}\bone$ is entrywise nonnegative.
\end{proposition}
\begin{proof}
By direct computation, the gradient over $\bw$ are   
\[
    \nabla\bw\trans M_2\bw = 2\bw\trans M_2 \text{ and } \nabla\bone\trans\bw = \bone\trans.
\]
Therefore, by the Lagrange multiplier theorem, the minimum happens when $k\bone = M_2\bw$ and $\bw = k M_2^{-1}\bone$ for some $k$.  Since $\bone\trans\bw = k\bone\trans M_2^{-1}\bone = m$, we have
\[
    k = \frac{m}{\bone\trans M_2^{-1}\bone}.
\]
The rest of the statements follow from direct computation.
\end{proof}

We now move on to consider the minimum variance associated with a graph.

\begin{definition}
Let $G$ be a graph on $n$ vertices and $m$ edges.  Define the \emph{approximated minimum variance} of $G$ as  
\[
    \amv(G) = 
    \frac{\bw\trans M_2\bw}{n-1} - \left(\frac{2m}{n-1}\right)^2 =
    \frac{m^2}{(n-1)\bone\trans M_2^{-1}\bone} - \left(\frac{2m}{n-1}\right)^2,
\]
where  
\[
    \bw = kM_2^{-1}\bone \text{ with }k = \frac{m}{\bone\trans M_2^{-1}\bone}.
\]
We say $G$ is \emph{eligible} if $M^{-1}\bone$ is entrywise positive.
\end{definition}

Thus, we know the $\amv(G) \leq \mv(G) \leq \onemv(G)$ for any connected graph $G$. 

\begin{proposition}
\label{prop:linereg}
Let $G$ be a graph on $n$ vertices and $m$ edges whose line graph is $r$-regular.  Then $G$ is eligible and  
\[
    \mv(G) = 
    \frac{m(4+r)}{n-1} - \left(\frac{2m}{n-1}\right)^2.
\]
\end{proposition}
\begin{proof}
When the line graph of $G$ is $r$-regular, we have $M_2\bone = (4 + r)\bone$, so $M^{-1}\bone = \frac{1}{4 + r}\bone$ and $G$ is eligible.  Thus, $\bone\trans M_2^{-1}\bone = \frac{m}{4 + r}$.  By direct computation, we have 
\[
    \amv(G) = 
    \frac{m^2}{(n-1)\bone\trans M_2^{-1}\bone} - \left(\frac{2m}{n-1}\right)^2 = 
    \frac{m(4+r)}{n-1} - \left(\frac{2m}{n-1}\right)^2.
\]
Moreover, $\mv(G) = \amv(G)$ by \cref{prop:mvsol}.
\end{proof}

Note that the line graph of $G$ is regular if and only if $G$ is either a regular graph or a biregular bipartite graph.  (A biregular bipartite graph is a bipartite graph such that the vertices on the same part have the same degree.)

\begin{corollary}
Let $G$ be a $k$-regular graph on $n$ vertices.  Then $G$ is eligible and  
\[
    \mv(G) = 
    \frac{nk}{n-1}\left(1 - \frac{k}{n-1}\right).
\]
In particular, $\mv(K_n) = 0$ and 
\[
    \mv(C_n) = \frac{2n}{n-1}\left(1 - \frac{2}{n-1}\right).
\]
\end{corollary}
\begin{proof}
Since $G$ is $k$-regular, it has $m = \frac{nk}{2}$ edges and its line graph is $r$-regular with $r = 2k - 2$.  Then the statement follows from  \cref{prop:linereg}.
\end{proof}

\begin{corollary}
Let $G$ be a $p,q$-biregular graph on $n$ vertices.  Then $G$ is eligible and  
\[
    \mv(G) = 
    \frac{m(p + q + 2)}{n-1} - \left(\frac{2m}{n-1}\right)^2 \text{ with }m = \frac{pq}{(p + q)^2}n^2.
\]
Consequently, $\mv(K_{1,n-1}) = n-2$.
\end{corollary}
\begin{proof}
Let $G$ be a $p,q$-biregular graph with its partition sizes $a$ and $b$.  By counting the number of edges, we have $ap = bq$.  Along with the condition $a + b = n$, we may solve that $a = \frac{q}{p + q}n$ and $b = \frac{p}{p + q}n$.  Thus, $m = \frac{pq}{(p + q)^2}n^2$ and the line graph of $G$ is $r$-regular with $r = p + q - 2$.  Then the statement follows from  \cref{prop:linereg}.
\end{proof}

The next two lemmas are essential for the study of $\mv(P_n)$.  Here we consider the $k\times k$ matrix  
\begin{equation}
\label{eq:am}
    A_k = \begin{bmatrix}
        4 & 1 & 0 & \cdots & 0 \\
        1 & 4 & 1 & \ddots & \vdots \\
        0 & 1 & \ddots & \ddots & 0 \\
        \vdots & \ddots & \ddots & ~ & 1 \\
        0 & \cdots & 0 & 1 & 4
    \end{bmatrix}
\end{equation}
Note that $A_m$ is the $M_2$ matrix of $P_n$ when $m = n - 1$.  

\begin{lemma}
\label{lem:amdet}
Let $A_k$ be as in \cref{eq:am} and $d_k = \det(A_k)$.  Then the following hold.  
\begin{enumerate}
\item $d_k = \frac{3+2\sqrt{3}}{6}(2+\sqrt{3})^k + \frac{3-2\sqrt{3}}{6}(2-\sqrt{3})^k$.
\item $d_{k+1} \geq (2 + \sqrt{3})d_k$ for any $k\geq 0$.  
\item $\frac{d_{k+1}}{d_k} \to 2 + \sqrt{3}$ as $k \to \infty$.   
\end{enumerate}
\end{lemma}
\begin{proof}
By the Laplace expansion along the first row of $A_k$, we have $d_k:= \det(A_k) = 4\det(A_{k-1}) - \det(A_{k-2})$.  Consequently, we have the recurrence relation  
\[
    \begin{aligned}
        d_{k+2} &= 4d_{k+1} - d_k, \\
        d_0 &= 1,\ d_1 = 4.
    \end{aligned}
\]
The characteristic polynomial of this recurrence relation is $r^2 - 4r + 1$, which as roots $2\pm\sqrt{3}$.  By solving the recurrence relation we obtain the general form of $d_k$.  Then the second and the third statements follow from the general form.
\end{proof}

\begin{lemma}
\label{lem:pathm2inv}
Let $A_k$ be as in \cref{eq:am} and $d_k = \det(A_k)$.  Let $M_2 = A_m$ for some fixed $m$.  Then the following hold.
\begin{enumerate}
\item $\det(M_2(i,j))$ equals $d_{i-1}d_{m-j}$ if $i \leq j$ and $d_{j-1}d_{m-i}$ if $i\geq j$.  
\item The $i$-th row sum of $M_2^{-1}$ is $\frac{1}{6}\left(1 + (-1)^{i+1}\frac{d_{m-i}}{d_m} + (-1)^{m-i}\frac{d_{i-1}}{d_m}\right)$.
\item $\bone\trans M_2^{-1}\bone = \frac{1}{6}m + \frac{1}{18} + \frac{d_{m-1}}{18d_m} + (-1)^{m-1}\frac{1}{18d_m}$.
\end{enumerate}
\end{lemma}
\begin{proof}
The formula of $\det(M_2(i,j))$ follows directly from the structure of $M_2$.  Let $w_{i,j}$ be the $i,j$-entry of $M_2^{-1}$.  Since $M_2^{-1}$ is symmetric, we have $w_{i,j} = (-1)^{i+j}\frac{\det(M_2(i,j))}{\det(M_2)}$.  In particular, $w_{1,i} = (-1)^{1+i}\frac{d_{m-i}}{d_m}$.

By \cite[Corollary~4.3]{dFP01}, the $i$-th row sum of $M_2$ is  
\[
    \frac{1 + w_{1,i} + w_{1,m-i+1}}{6} = \frac{1}{6}\left(1 + (-1)^{i+1}\frac{d_{m-i}}{d_m} + (-1)^{m-i}\frac{d_{i-1}}{d_m}\right).
\]

Let $s$ be the $1$-st row sum of $M_2^{-1}$.  By \cite[Corollary~4.4]{dFP01},  
\[
    \bone\trans M_2^{-1}\bone = \frac{m + 2s}{6} = \frac{1}{6}m + \frac{1}{18} + \frac{d_{m-1}}{18d_m} + (-1)^{m-1}\frac{1}{18d_m}.
\]
Here $s$ is the $1$-st row sum of $M_2^{-1}$.
This completes the proof.
\end{proof}

\begin{theorem}
\label{thm:pathamv}
The graph $P_n$ is eligible and $\mv(P_n)$ tends to $2$ as $n \to \infty$. 
\end{theorem}
\begin{proof}
By \cref{lem:pathm2inv} and the fact $d_m > 3d_{m-1}$ from \cref{lem:amdet}, we have $|\frac{d_{m-1}}{d_m}| < \frac{1}{3}$, $|\frac{d_{i-1}}{d_m}| < \frac{1}{3}$, and each row sum of $M_2^{-1}$ is positive, so $M_2^{-1}\bone$ is entrywise positive.   
 Thus, $P_n$ is eligible and $\mv(P_n) = \amv(P_n)$ by \cref{prop:mvsol}.

From the formula of $\bone\trans M_2^{-1}\bone$ in \cref{lem:pathm2inv} and $m = n - 1$, we have  
\[
    \begin{aligned}
        \mv(P_n) &= 
    \frac{m^2}{(n-1)\bone\trans M_2^{-1}\bone} - \left(\frac{2m}{n-1}\right)^2 \\
        &= \frac{m}{\bone\trans M_2^{-1}\bone} - 4.
    \end{aligned}
\]
Note that $\bone\trans M_2^{-1}\bone \sim \frac{1}{6}m + o(m)$.  Hence  
\[
    \amv(P_n) \to \frac{m}{\frac{1}{6}m} - 4 = 2
\]
as $n \to \infty$.  
\end{proof}

Now we move our attention to $\onemv(G)$.  
\begin{proposition}
\label{prop:onemv}
Let $G$ be a connected graph on $n$ vertices and $m$ edges.  Let $d_1, \ldots, d_n$ be its degree sequence.  Then 
\[
    \onemv(G) = \frac{2m + \sum_{i=1}^n d_i^2}{n-1} - \left(\frac{2m}{n-1}\right)^2. 
\]
\end{proposition}
\begin{proof}
Let $M_2 = 4I + B$, where $B$ is the adjacency matrix of the line graph of $G$.  The line graph of $G$ has $\sum_{i=1}^n\binom{d_i}{2}$ edges, so 
\[
    \bone\trans M_2\bone = 4m + 2\sum_{i=1}^n\binom{d_i}{2} = 4m + \sum_{i=1}^n d_i^2 - \sum_{i=1}^n d_i = 2m + \sum_{i=1}^n d_i^2.
\]
Then the desired formula follows from \cref{prop:m2formula}.
\end{proof}

\begin{theorem}
Over all trees $T$ on $n$ vertices, $\onemv(T)$ is uniquely maximized by the star with $\onemv(T) = n-2$ and is uniquely minimized by the path.  

For general graphs $G$, $\onemv(G) \leq \frac{m}{n-1}\left(n - \frac{2m}{n-1}\right)$.
\end{theorem}
\begin{proof}
For trees $T$ on $n$ vertices,  
\[\onemv(T) = 2 + \frac{\sum_{i=1}^n d_i^2}{n-1} - 2^2\]
by \cref{prop:onemv}, so the quantity depends only on $\sum_{i=1}^n d_i^2$.  

We claim that the sum of squares of degrees is maximized by the star and minimized by the path among all trees on $n$ vertices.  To see this, pick a leaf $i$ of $T$ with its unique neighbor $j$ and pick some vertex $k$ other than $i$ and $j$.  Then $T' = T - \{i,j\} + \{i,k\}$ is again a tree.  If $d_j \leq d_k$, then the sum of squares of degrees of $T'$ is strictly larger than that of $T$.  Repeating this operation with $i$ an arbitrary leaf and $k$ a vertex of maximum degree, the sum is increasing throughout the process, which terminates only when the tree becomes a star.  On the other hand, if $d_j = d_k + 1$, then the sum of squares of degrees of $T'$ equals that of $T$; and if $d_j \geq d_k + 2$, then the sum of squares of degrees of $T'$ is strictly less than that of $T$.  Repeating the operation with $i$ an arbitrary leaf and $k$ an endpoint of the longest path, the sum reaches its unique minimum when the tree is a path. 

For general graph, there are several bounds for the sum of squares of degrees in \cite{Das04}.  In particular, de Caen \cite{dCaen98} showed that  
\[
    \sum_{i=1}^n d_i^2 \leq m\left(\frac{2m}{n-1} + n - 2\right),
\]
which gives the stated upper bound of $\onemv(G)$.
\end{proof}

\subsection{Algorithms for Finding the Minimum Variance}
\label{subs:qprog}

For this subsection, we provide algorithms to solve the minimum variance of a graph or, equivalently, \cref{pbm:m2} below.

\begin{problem}
\label{pbm:m2}
Let $G$ be a graph on $m$ edges.  Let $M_2 = 4I + B$, where $B$ is the adjacency matrix of the line graph of $G$.  Find the minimum of
\[
    f(\bw) = \bw\trans M_2\bw
\]
in the region 
\[
    R = \{\bw\in\mathbb{R}^{E(G)}: \bone\trans\bw = m \text{ and } \bw \geq 0\}.
\]
\end{problem}

Recall that a function $g$ is \emph{concave} (upward) if  
\[
    g((1-s)\bx + s\by) \geq (1-s)g(\bx) + sg(\by)
\]
for any $\bx,\by$ in the domain of the function and for any $s\in [0,1]$.  The following proposition is well-known in the subject on quadratic programming.

\begin{proposition}
\label{prop:m2concave}
If $M_2$ is a positive definite matrix, then the function $f(\bw) = \bw\trans M_2\bw$ is concave.  Moreover, there is a unique local minimum in any compact and convex region.
\end{proposition}
\begin{proof}
We first show the concavity.  Let $\bx$ and $\by$ be two different points in the region.  Define $g(s) = f((1-s)\bx + s\by)$ with $s\in[0,1]$, which can be expanded as  
\[
    g(s) = f(\bx + s(\by - \bx)) = \bx\trans M_2\bx + 2\bx\trans M_2(\by - \bx)s + (\by - \bx)\trans M_2(\by - \bx) s^2.
\]
Thus, $g(s)$ is a quadratic function in $s$ with a positive quadratic term since $M_2$ is positive definite, which is a concave function.  

A continuous function attains a global minimum at $\bx$ on a compact region.  For any point $\by \neq \bx$ in the domain, the segment $L$ connecting $\bx$ and $\by$ is also in the region by its convexity.  Note that the function $f$ restricted on $L$ is a quadratic function with positive quadratic term.  Since $\bx$ is a global minimum, the function $f$ restricted on $L$ must be an strictly increasing function from $\bx$ to $\by$.  Therefore, $\by$ cannot be a local minimum, and $\bx$ is the unique local minimum.
\end{proof}

With \cref{prop:m2concave} applied to our case, it sheds light on two facts:  First, if a point is verified as a local minimum by its derivatives, then it is also the global minimum.  Second, numerical methods for finding a local minimum, such as gradient descent, are guaranteed to find the global maximum.  These facts lead to algorithms for finding the minimum, where one is exact and the other is numerical.  

Define the \emph{support} of a vector as the set of indices on which the vector is nonzero.  We first characterize the conditions for a vector $\bw$ with $\supp(\bw) = \alpha$ being a local minimum.


\begin{proposition}
\label{prop:supplocal}
Let $m$, $M_2$, $f$, and $R$ be as in \cref{pbm:m2}.  Then $\bw$ is a local minimum in $R$ with the $\supp(\bw) = \alpha$ if and only if   
\begin{enumerate}
    \item $M_2[\alpha]^{-1}\bone$ is entrywise positive and 
    \item $M_2(\alpha]M_2[\alpha]^{-1}\bone$ is entrywise greater than or equal to $1$.
\end{enumerate}
If these conditions hold, then $\bw[\alpha] = kM_2[\alpha]^{-1}\bone$ with $k = \frac{m}{\bone\trans M_2[\alpha]^{-1}\bone}$ and $\bw(\alpha) = \bzero$ is the only local minimum with support $\alpha$.
\end{proposition}
\begin{proof}
By \cref{prop:m2concave}, $f$ becomes a quadratic function with positive quadratic term on any segment.  Therefore, a point is a local minimum if and only if the directional derivative at this point is positive or zero for any possible direction.  

At $\bw$, the gradient of $f$ is  
\[
    \nabla f(\bw) = 2\bw\trans M_2.
\]
That is, given any unit vector $\bd$, the directional derivative of $f$ along $\bd$ is  
\[
    D_\bd f(\bw) = 2\bw\trans M_2 \bd.
\]
Since we only care about the signs of $D_\bd f(\bw)$, we know that $\bw$ is a local minimum if and only if $\bw\trans M_2\bd \geq 0$ for any possible direction $\bd$.  

Since  $\supp(\bw) = \alpha$, the vectors $\bd = \be_i - \be_j$ for any ($i,j\in\alpha$) or ($i\notin\alpha$ and $j\in\alpha$) are some possible directions for $\bw$ to move around $R$.  In fact, all possible directions are the linear combinations of these vectors using nonnegative coefficients.  Consequently, $\bw$ is a local minimum if and only if $\bw\trans M_2\bd \geq 0$ for these vectors $\bd$.  

By choosing $\bd$ as $\be_i - \be_j$ and $\be_j - \be_i$ with $i,j\in\alpha$, we know the entries of $\bw\trans M_2$ on $\alpha$ are the same.  By choosing $\bu$ as $\be_i - \be_j$ with $i\notin\alpha$ and $j\in\alpha$, we know the entries of $\bw\trans M_2$ outside $\alpha$ are greater than or equal to those on $\alpha$.  Since $\bw$ is zero outside $\alpha$, we have  
\[
    \bw\trans M_2 = \bw\trans M_2 = 
    \begin{bmatrix}
        \bw[\alpha]\trans M_2[\alpha] &
        \bw[\alpha]\trans M_2[\alpha)
    \end{bmatrix}.
\]
In summary, the entries of $\bw[\alpha]\trans M_2[\alpha] = k\bone$ for some $k$, while each entry of $\bw[\alpha]\trans M_2[\alpha)$ is greater than or equal to $k$.

Since $M_2[\alpha]$ is a principal submatrix of a positive definite matrix $M_2$, $M_2[\alpha]$ is invertible.  Thus, $\bw[\alpha]\trans M_2[\alpha] = k\bone$ implies $\bw[\alpha] = k M_2[\alpha]^{-1}\bone$.  Also, $\bw(\alpha) = \bzero$ as $\supp(\bw) = \alpha$.  Thus, we may derive $k = \frac{m}{\bone\trans M_2[\alpha]^{-1}\bone} > 0$ from $\bone\trans\bw = m$.  With this formula of $\bw$, it is a local minimum if and only if $M_2[\alpha]^{-1}\bone$ is entrywise positive, which ensures $\supp(\bw) = \alpha$, and $M_2(\alpha] \bw[\alpha]$ is entrywise larger than or equal to the constant value on $M_2[\alpha] \bw[\alpha]$, which ensures $\bw$ is a local minimum.  As $\bw[\alpha] = k M_2[\alpha]^{-1}\bone$, the second condition is equivalent to $M_2(\alpha]\cdot kM_2[\alpha]^{-1}\bone$ is entrywise greater than or equal to $k$, where $k$ can be canceled on both sides. 
\end{proof}

Based on \cref{prop:supplocal}, we may define $\bw_\alpha$ such that whose entries on $\alpha$ match the vector $kM_2[\alpha]^{-1}\bone$ with $k = \frac{m}{\bone\trans M_2[\alpha]^{-1}\bone}$ and otherwise zero.  We say $\alpha$ is \emph{eligible} for $G$ if $\alpha$ satisfies the two conditions in \cref{prop:supplocal}.

\begin{remark}
Combining \cref{prop:m2concave}, there is a unique $\alpha\in [m]$ that is eligible for $G$.  Thus, if some eligible $\alpha\subseteq [m]$ is found, then $\bw_\alpha$ must be the global minimum, and we do not have to check others.  Meanwhile, by checking all possible $\alpha$, we found the minimum value of $f$, which provides an exact, yet exhaustive, algorithm for searching the minimum.  

Unfortunately, this requires checking all $2^m - 1$ possible $\alpha$, as $\alpha\neq\emptyset$.  Recall that for linear programming in $\mathbb{R}^n$ with $r$ restrictions, one also needs to check $\binom{r}{n}$ exhaustively, so it might not be easy to improve it further. 
\end{remark}

\begin{algorithm}
\label{alg:allalpha}
This algorithm solves \cref{pbm:m2}.

\textbf{Input}: a graph $G$

\textbf{Output}: the $\bw$ that minimizes $\bw\trans M_2\bw$ with $\bone\trans\bw = m$ and $\bw \geq 0$.

\begin{enumerate}
    \item For each $\alpha\subseteq [m]$ and $\alpha\neq\emptyset$:
        \begin{enumerate}
            \item Check if $M_2[\alpha]^{-1}\bone$ is entrywise positive.
            \item Check if $M_2(\alpha]M_2[\alpha]^{-1}\bone$ is entrywise greater than or equal to $1$.
            \item If both conditions hold, then keep the eligible $\alpha$ and stop searching.
        \end{enumerate}
    \item Define $\bw[\alpha] = kM_2[\alpha]^{-1}\bone$ with $k = \frac{m}{\bone\trans M_2[\alpha]^{-1}\bone}$ and $\bw(\alpha) = \bzero$.
\end{enumerate}

\end{algorithm}

\begin{example}
Let $G$ be the double star obtained from $K_{1,p}$ and $K_{1,q}$ by adding an edge joining their centers. 
Observe that  
\[
    M_2 = \begin{bmatrix}
        J_p + 3I_p & \bone_p & O \\
        \bone_p\trans & 4 & \bone_q\trans \\
        O & \bone_q & J_q + 3I_q
    \end{bmatrix}.
\]
Then by direct computation, we have  
\[
    \bw_{[m]} = M_2^{-1}\bone = k_1\begin{bmatrix}
        (q + 3)\bone_p \\
        \frac{9 - pq}{3} \\
        (p + 3)\bone_q
    \end{bmatrix} 
    \text{ with }
    k_1 = \frac{p + q + 1}{\frac{5}{3}pq + 3p + 3q + 3}.
\]
This means, $G$ is eligible if and only if $pq < 9$.  On the other hand, let $\alpha = E(K_{1,p}) \cup E(K_{1,q})$.  By direct computation, we have  
\[
    M_2[\alpha]^{-1}\bone_{p+q} = 
    \begin{bmatrix}
        \frac{1}{p + 3}\bone_p \\
        \frac{1}{q + 3}\bone_q,
    \end{bmatrix}
\]
which is entrywise positive.  Also,
\[
    M_2(\alpha]M_2[\alpha]^{-1}\bone = 
    \frac{p}{p + 3} + \frac{q}{q + 3} = 
    \frac{2pq + 3p + 3q}{pq + 3p + 3q + 9}
\]
is entrywise greater than or equal to $1$ whenever $pq \geq 9$.  In this case, $\alpha$ is eligible for $G$ and   
\[
    \bw_\alpha = k_2\begin{bmatrix}
        (q + 3)\bone_p \\
        0 \\
        (p + 3)\bone_q
    \end{bmatrix} 
    \text{ with }
    k_2 = \frac{p + q + 1}{2pq + 3p + 3q}.
\]
Therefore, the minimum variance of $G$ is achieved by $\bw_{[m]}$ when $pq < 9$ and by $\bw_\alpha$ otherwise.
\end{example}

Next, we establish the gradient descent algorithm designed for solving \cref{pbm:m2}.  Since $f$ is quadratic along any direction, we may calculate precisely how far to go on each direction to decrease the function value.  

\begin{observation}
\label{obs:qformmin}
Given a point $\bw$ and a direction $\bd$, we may compute
\[
    f(\bw + t\bd) = \bw\trans M_2\bw + 2\bw\trans M_2\bd t + \bd\trans M_2\bd t^2.  
\]
The minimum of this function, with respect to $t$, happens at $t = \frac{-\bw\trans M_2\bd}{\bd\trans M_2\bd}$.  In the case when $\bd = \be_i - \be_j$, we have $-\bw\trans M_2\bd = (\bw\trans M_2)_j - (\bw\trans M_2)_i$ and $\bd\trans M_2\bd$ is $6$ or $8$, depending on edge $i$ and edge $j$ are incident or not.  
\end{observation}

\begin{algorithm}
\label{alg:gd}
This algorithm solves \cref{pbm:m2}.

\textbf{Input}: a graph $G$

\textbf{Output}: the $\bw$ that minimizes $\bw\trans M_2\bw$ with $\bone\trans\bw = m$ and $\bw \geq 0$.

\begin{enumerate}
    \item Start with $\bw = \bone\in\mathbb{R}^{E(G)}$.  
    \item Compute $M_2\bw$ and let $i$ be the index of its minimum entry.
    \item For each $j$, compute $t_j = \frac{(M_2\bw)_j - (M_2\bw)_i}{8}$.
    \item Choose $j$ so that $\eta_j = \min\{t_j, (\bw)_j\}$ is maximized.
    \item Let $\bd = \be_i - \be_j$ and update $\bw \leftarrow \bw + \eta_j\cdot \bd$.  
    \item Repeat from Step~2 as needed.
\end{enumerate}
\end{algorithm}

\begin{theorem}
\cref{alg:gd} generates a sequence of $\bw$ that converges to the minimum of $f(\bw) = \bw\trans M_2\bw$ subject to the given conditions.
\end{theorem}
\begin{proof}
Let $\bw_0 = \bone$ and $\bw_n$ be the $\bw$ vector in each iteration in \cref{alg:gd}.  By \cref{prop:m2concave}, there is a unique local minimum $\bw_{\min}$ in the region.  We will show that $\bw_n$ converges to $\bw_{\min}$.  

First, we define a function  
\[
    \eta(\bw) = \max_j \min\left\{
    \frac{(M_2\bw)_j - \min(M_2\bw)}{8}, (\bw)_j
    \right\},
\]
where $\min(M_2\bw)$ is the minimum entry of $M_2\bw$.  By definition, $\eta(\bw) \geq 0$.  Moreover, $\eta(\bw) = 0$ if and only if $\bw$ is a local minimum by \cref{prop:supplocal}, which means $\bw = \bw_{\min}$.  Also, we see that the algorithm updates $\bw$ by $\bw_{n+1} = \bw_n + \eta(\bw_n)\bd$.  According to the formula in \cref{obs:qformmin}, $f(\bw_n)$ is decreasing whenever $\eta(\bw_n) > 0$; moreover,   
\[
    |f(\bw_{n+1}) - f(\bw_n)| \geq 6\eta(\bw_n)^2
\]
since $\bd\trans M_2\bd \geq 6$.

We claim that $\eta(\bw_n) \to 0$ as $n\to\infty$.  Given any $\epsilon > 0$, if there are infinitely many $\bw_n$ such that $\eta(\bw_n) \geq \epsilon$, then $f(\bw_n)$ would drop by at least $6\epsilon^2$ for infinitely many times, violating the fact that $f(\bw) \geq 0$.  Therefore, there are only finitely many $\bw_n$ with $\eta(\bw_n) \geq \epsilon$, which is equivalent to $\eta(\bw_n) \to 0$.  

Since $\bw_n$ is a sequence on the compact region, there is a convergent subsequence $\bw_{n_i}$ that converges to $\bw_\infty$.  By the previous discussion, $\eta(\bw_\infty) = 0$, which means $\bw_\infty = \bw_{\min}$.  Therefore, $f(\bw_n)$ is decreasing and converges to the minimum.  

Finally, we claim that not just $\bw_{n_i}$, but also $\bw_n$, converges to $\bw_{\min}$.  To see this, pick an arbitrary unit vector $\bd$ such that $\bw_{\min} + t\bd$ is in the region for any small $t > 0$.  By \cref{obs:qformmin}, we have 
\[
    f(\bw_{\min} + t\bd) \geq f(\bw_{\min}) + 2t^2
\]
since the minimum eigenvalue of $M_2$ is at least $2$ and $\bd\trans M_2\bd \geq 2$.  This means $\|\bw - \bw_{\min}\| \geq t$ implies $f(\bw) - f(\bw_{\min}) \geq  2t^2$.  Since $f(\bw_n) \to f(\bw_{\min})$ as $n \to \infty$, we also have $\bw_n \to \bw_{\min}$.  
\end{proof}

\section{Conclusions and Future Considerations}

In this work, we have focused our attention on the possible spectra allowed for generalized Laplacian matrices associated with a given graph. For graphs on at most four vertices we derived results on potential spectra and allowed ordered multiplicity lists for generalized Laplacians. For the special case of complete graphs and stars we have general claims concerning the possible spectra of generalized Laplacians associated with these two families of graphs. In Section 4, we adjusted our focus and considered certain moments associated with the eigenvalues of normalized weightings of generalized Laplacian matrices of graphs. A key element in this work was the minimum variance of a graph corresponding to matrices in $\mptn_L^s(G)$. We then proceed to either compute or bound the minimum variance of specific families of graphs, including regular graphs and trees, and develop tools from quadratic programming to determine the minimum of a related quadratic form associated with a matrix (labeled $M_2$) of a graph in this context.

We close this section with a list of related potentially interesting directions concerning the spectra of generalized Laplacian matrices associated with a graph.

\begin{enumerate}
    \item Concerning the Laplacian spectra of small graphs, we ask is the boundary for $\Paw, C_4$, and $K_4 - e$ in \cref{fig:g4} linear? 
\item We have considered describing the allowed spectrum over all matrices $\mptn_L(G)$ for given graphs or for specified families of graphs. A natural specialization along these lines is, for a given graph $G$,  to consider computing the maximum possible multiplicity over the matrices $\mptn_L(G)$ and to minimize  the number of distinct eigenvalues over all matrices in $\mptn_L(G)$.
 \item Finally, strong matrix properties (such as the Strong Arnold Property or the 
 Strong Spectral Property, see  \cite{gSAP}) have become very important for the inverse eigenvalue problem for graphs. Thus, it seems reasonable to study potential strong properties of matrices in $\mptn_L(G)$ and their applications to matrix perturbations and possible spectral properties.

\end{enumerate}

\vspace{.5cm}

\subsection*{Acknowledgments}
S.M. Fallat was supported in part by an NSERC Discovery Research Grant, Application No.: RGPIN-2019-03934. 
The work of the PIMS Postdoctoral Fellow H. Gupta leading to this publication was supported in part by the Pacific Institute for the Mathematical Sciences.  
J.C.-H. Lin was supported by the National Science and Technology Council of Taiwan (grant no.\ NSTC-112-2628-M-110-003 and grant no.\ NSTC-113-2115-M-110-010-MY3).


\end{document}